\documentclass{amsart}

\usepackage{amsmath,amsthm,amssymb,amsfonts}
\usepackage{color}

\newcommand{\R}{\mathbb{R}}

\newcommand{\un}{\mathbf{1}\!\!{\rm I}}
\newcommand{\mn}{|\!\!|}
\newcommand{\mnp}{|\!\!|_{M^{d/2}}}
\newcommand{\be}{\begin{equation}}
\newcommand{\ee}{\end{equation}}
\newcommand{\bea}{\begin{eqnarray}}
\newcommand{\eea}{\end{eqnarray}}
\newcommand{\bean}{\begin{eqnarray*}}
\newcommand{\eean}{\end{eqnarray*}}
\newcommand{\rf}[1]{(\ref {#1})}

\def\dx{\,{\rm d}x}
\def\dy{\,{\rm d}y}
\def\dt{\,{\rm d}t}

\def\dr{\,{\rm d}r}

\def\e{\varepsilon}

\def\r{\varrho}
\def\xn{|\!|\!|}
\def\tbl{T_{\rm b}}

\newtheorem{theorem}{Theorem}
\newtheorem{proposition}[theorem]{Proposition}
\newtheorem{lemma}[theorem]{Lemma}
\newtheorem{corollary}[theorem]{Corollary}

\theoremstyle{definition}
\newtheorem{definition}[theorem]{Definition}

\theoremstyle{remark}
\newtheorem{remark}[theorem]{Remark}

\def\int{\intop\limits}

\numberwithin{equation}{section}
\numberwithin{theorem}{section}

\author{Piotr Biler}
\address{\small Instytut Matematyczny, Uniwersytet Wroc\l awski,
 pl. Grunwaldzki 2/4, 50-384 Wroc\-\l aw, Poland}
\email{Piotr.Biler@math.uni.wroc.pl}

\author{Grzegorz Karch}
\address{\small
 Instytut Matematyczny, Uniwersytet Wroc\l awski,
 pl. Grunwaldzki 2/4, 50-384 Wroc\-\l aw, Poland}
\email{Grzegorz.Karch@math.uni.wroc.pl}
\urladdr{http://www.math.uni.wroc.pl/~karch}

\author{Jacek Zienkiewicz}
\address{\small
 Instytut Matematyczny, Uniwersytet Wroc\l awski,
 pl. Grunwaldzki 2/4, 50-384 Wroc\-\l aw, Poland}
\email{Jacek.Zienkiewicz@math.uni.wroc.pl}

\title[blowup in chemotaxis]{Optimal criteria for blowup of radial and \\$N$-symmetric solutions of chemotaxis systems}

\begin{document}

\begin{abstract}
A simple proof of  concentration of mass equal to $8\pi$ for blowing up  $N$-symmetric solutions of the Keller--Segel model of chemotaxis in two dimensions with large $N$ is given. Moreover, a criterion for blowup of solutions in terms of the radial initial concentrations, related to suitable Morrey spaces norms, is derived for radial solutions of chemotaxis in several dimensions. This condition is, in a sense,  complementary to the one guaranteeing the global-in-time existence of solutions.
\end{abstract}



\date{\today}

\thanks{The authors thank the referees for many pertinent remarks that permitted them to improve the presentation of results. The preparation of the paper was supported by the NCN grant  2013/09/B/ST1/04412 (the first and the second authors). The third author was also supported by the grant  DEC-2012/05/B/ST1/00692.}

\maketitle

\baselineskip=20pt

\section{Introduction}
We consider in this paper the classical parabolic-elliptic Keller--Segel model of chemotaxis in $d\ge 2$ space dimensions
\bea
u_t-\Delta u+\nabla\cdot(u\nabla v)&=&0,\label{equ}\\
\Delta v+u &=& 0,\label{eqv}
\eea
supplemented with a nonnegative initial condition
\be
u(x,0)=u_0(x)\ge 0\label{ini}.
\ee
Here for $(x,t)\in {\mathbb R}^d\times[0,T)$, the function $u=u(x,t)\ge 0$ denotes the density of the population of microorganisms, $v=v(x,t)$ --- the density of the chemical secreted by themselves that attracts them and makes them to aggregate. The system \rf{equ}--\rf{eqv} also models the gravitational attraction of particles in astrophysical models, see \cite{B-SM, B-CM}. 

As it is well known,  cf. e.g. \cite{BDP,BM}, the total mass of the initial condition
\be
M=\int_{\R^d} u_0(x)\dx\label{M}
\ee
is conserved in time, i.e.
$ \int_{\R^d} u(x,t)\dx = \int_{\R^d} u_0(x)\dx =M$
for all $t\in [0,\tbl)$, and this 
is the critical quantity for the global-in-time existence of nonnegative solutions in the two-dimensional case. Namely, if $M\le 8\pi$, then solutions of \rf{equ}--\rf{ini} (with $u_0$ --- a~finite nonnegative measure)  exist for all $t\ge 0$. For the local-in-time existence, it should be assumed that all the  atoms of the finite measure $u_0$ are of mass less than $8\pi$, see \cite{BDP,BM,BZ}. 
  
When $M>8\pi$, nonnegative solutions blow up in a finite time, and  for radially symmetric solutions  mass equal to $8\pi$ concentrates at the origin at the blowup time, see \cite[Ch. 11]{S}, \cite{B-BCP}, resp.
The multidimensional case is different: there are solutions of the chemotaxis system with arbitrarily small $M>0$ that cease to exist after a finite time elapsed, see for instance 
\cite{B-CM, HV, BK-JEE}. 

Global-in-time solutions of the  chemotaxis system have been constructed in various functional spaces (including,  e.g., finite Radon measures if $d=2$, and Lebesgue $L^p$, Marcinkiewicz weak $L^p$ spaces with  $p=d/2$, Besov and Morrey spaces if $d\ge 3$),  cf. for instance \cite{BM, B-SM, Lem}, under smallness conditions on norms of $u_0$ in a relevant space which is critical for \rf{equ}--\rf{eqv}. Here ``critical'' means that solutions obey  the same scaling property as  the norms in those spaces, cf. \cite{Lem} for more detailed explanations.

First, we show in the present work that the radial concentration of  data is the critical quantity for the finite time blowup of nonnegative radial solutions of \rf{equ}--\rf{ini}. Here, we define {\em the radial concentration} by the formula
\be
\xn u_0\xn\equiv \sup_{R>0}R^{2-d}\int_{\{|x|\le R\}} \psi\left(\frac{x}{R}\right)u_0(x)\dx \label{conc}
\ee
with a fixed radial nonnegative, piecewise ${\mathcal C}^2$ function $\psi$ supported on the unit ball, such that $\psi(0)=1$. Clearly, those quantities for such weight functions $\psi$ are comparable, so we fix in the following $\psi(x)=\left(1-|x|^2\right)^2_+$, see \rf{psi} below. 

Here, one should recall that the homogeneous Morrey space $M^p(\R^d)$, $1<p<\infty$,   is defined as the set of all locally integrable functions $f$ that satisfy
\be
\mn f\mn_{M^p}=\sup_{x_0\in\R^d,\,R>0}R^{d(1/p-1)}\int_{\{|y-x_0|\le R\}} |f(y)|\dy<\infty. \label{norm}
\ee
Of course, for $d\ge 3$ and  $p=d/2$, the norm in $M^{d/2}(\R^d)$ (relevant to the theory of existence of local-in-time solutions) dominates the radial concentration \rf{conc}: $\mn u_0\mnp\ge R^{2-d}\int_{\{|x|\le R\}}u_0(x)\dx$ \ for each $R>0$,  but, in fact, for radially symmetric $u_0$ both  quantities $\mn u_0\mn_{M^{d/2}}$ and $\xn u_0\xn$ are equivalent. 

The criticality of the radial concentration \rf{conc}
means  that for initial data with small $\xn u_0\xn$ solutions exist indefinitely in time (cf. \cite{B-SM, Lem}), while our result stated in  Theorem \ref{radi} below  shows that for sufficiently big $\xn u_0\xn$ regular solutions cease to exist in a finite time. 
This result seems to be new for $d\ge 3$ although related criteria appeared in,  e.g., \cite{B-CM} and \cite{BK-JEE}. They have been, however, formulated in terms of ``global  quantities'' like the second moment $\int |x|^2  u_0(x)\dx$ while \rf{conc} is a local quantity, and its definition does not require supplementary properties of $u_0$ like 
 $\int |x|^2u_0(x)\dx<\infty$.

The proof of our first  result (contained in the following theorem) 
on the occurrence of radially symmetric blowup for $d\ge 2$ does not involve ``global quantities'', and  its idea is astonishingly simple.

\begin{theorem} [Blowup of radial  solutions with large concentration] \label{radi}
For each $d\ge 2$ there exists a constant $C_d>0$ such that if $u_0\in L^1(\R^d)$ is a radially symmetric function and $R^{2-d}\int_{\{|x|\le R\}}\psi\left(\frac{x}{R}\right)u_0(x)\dx>C_d$ for some $R>0$, then the solution $u$ of  problem \rf{equ}--\rf{ini} blows up in a finite time.
\end{theorem}

\begin{remark}
Note that for $d=2$ we recover the well-known result: if $M>C_2=8\pi$, then the solution of \rf{equ}--\rf{ini} blows up in a finite time, see the end of the proof of Theorem \ref{radi} below. 
In fact, this proof (involving a local moment  of the solution) extends to the two-dimensional case ($x\in \mathbb R^2$) of arbitrary (not necessarily radially symmetric) nonnegative solutions,  cf.  \cite{N,KO,BCKZ}  for similar arguments.  
Some improvements of the results in Theorem  1.1 (with direct relations to the critical values of Morrey norms and with quite different proofs) are in~\cite{BKZ}.
\end{remark}

In our second result, we limit ourselves to the Cauchy problem \rf{equ}--\rf{ini} on the plane, and we show a concentration  at the origin of mass equal {\em exactly} to $8\pi$ for some solutions.
This result is known in the radially symmetric case;  our proof 
allows us to deal with a larger class of solutions, and is conceptually much simpler than existing ones.  
An analysis of the chemotaxis system in bounded planar domains leads to blowups at interior and boundary points, see \cite[Th. 1.1]{S}, and to the quantization of mass $8k\pi$, $k\in \mathbb N$, at the interior blowup points for solutions with finite free energy, see \cite[Th. 1.2, Th. 15. 1]{S}. 
  The proofs  of those results in \cite[Ch. 11--15]{S} rely on subtle estimates of the free energy for \rf{equ}--\rf{eqv}  and various functional inequalities of Gagliardo--Nirenberg--Sobolev type. For an earlier approach to different versions of  conjectures on concentration we refer the reader to \cite[p. 23--24]{SS}. 
Compare  also  \cite{B-BCP} for  the radially symmetric case. 

Our proof for {\em radially symmetric} solutions with 
\be
u_0(z{\rm e}^{i\vartheta})=u_0(z), \ \ \vartheta\in\mathbb R,\label{rad}
\ee 
extends to the case of {\em $N$-symmetric} solutions with sufficiently large $N$,  i.e.  those with the initial data satisfying 
\be
u_0\left(z{\rm e}^{ik\frac{2\pi}{N}}\right)=u_0(z),\ \ \ k\in{\mathbb  N},\label{sym}
\ee
with the natural identification $\R^2\ni x\leftrightarrow z\in \mathbb C$.
Note that by the uniqueness of nonnegative solutions, the solution $u(x,t)$ with $u_0$ satisfying \rf{rad} is radial, and for that satisfying \rf{sym}\ \  $u$ is $N$-symmetric for each admissible $t$: $u\left(z{\rm e}^{i\frac{k}{N}\pi},t\right)=u(z,t)$.  
The interest in such solutions is related to the problem of study of 
certain bilinear integrals involving derivatives of the  fundamental solution of Laplacian.

Moreover, we are motivated by results in \cite{SSV}, where 2-symmetric distributions have been considered, see Remark \ref{rem:Vel} below for more information.

\begin{theorem}[Blowup with $8\pi$ concentration of mass]\label{symm} 
Assume that the initial condition $0\le u_0\in L^1(\R^2)$ is  $N$-symmetric in the sense of equation \eqref{sym} and such that $\int u_0(x)\dx=M>8\pi$.
Let $u(x,t)$, for $x\in \R^2$ and $t< \tbl$, be the corresponding classical solution of problem \rf{equ}--\rf{ini} which cannot be continued past the blowup time $t=\tbl$. 
If $N$ is sufficiently large so that $M/N$ is small enough,  then
$$\limsup_{t \rightarrow \tbl,\ R\to 0}\int_{\{|x|\leq R\}} u(x,t)\dx \leq 8\pi.$$
\end{theorem}

\par\noindent
In fact, $N$-symmetric solutions blow up with the concentration of mass {\it equal} to $8\pi$.

\begin{corollary}\label{cor:symm}
Under the assumptions of Theorem \ref{symm}, if moreover, $\int u_0(x) |x|^2\dx <\infty$, then
$$\lim_{R\to 0}\lim_{t \rightarrow \tbl}\int_{\{|x|\leq R\}} u(x,t)\dx =8\pi.$$
\end{corollary}

 In the proof of Theorem \ref{symm}, we use   simple (but rather subtle) techniques of weight functions and scalings. The core of our analysis consists in  uniform (with respect to initial data) estimates on a blowup time (see  Proposition \ref{genbl}) and on the uniform spread (or decay) of mass for symmetric initial conditions (see Proposition~\ref{symmass}).
The proofs of these two propositions is much shorter  is the radially symmetric case, which we emphasize below.
Moreover, we use systematically the well-known rescaling of the system:
for each $\lambda>0$ and each solution $u$ of \rf{equ}--\rf{eqv} of mass $M$ the function 
\be
u_\lambda(x,t)=\lambda^2u(\lambda x,\lambda^2t)\label{scale}
\ee 
is also a solution, with its mass  equal to $M$. 

Corollary \ref{cor:symm} is  a direct consequence of Theorem \ref{symm} combined with results proved in \cite{BZ}, see the end of Section \ref{3}.

\begin{remark}\label{rem:Vel}
The authors of \cite{SSV} suggested how to construct solutions of the Keller--Segel system that blow up with  the {\em quantized} concentration of mass $M=16\pi$. 
In view of Theorem \ref{symm}, their data cannot be $N$-symmetric with large $N\gg 2$.
\end{remark}

\medskip

\subsection*{Notations.}
The integrals with no integration limits are understood as over the whole space $\R^d$: \  $\int=\int_{\R^d}$. The letter $C$ denotes various constants which may vary from line to line but they are independent of solutions. As usual, $\sigma_d=\frac{2\pi^{d/2}}{\Gamma(d/2)}$ denotes the area of the unit  sphere in $\R^d$.

\section{Proof of blowup of radial solutions}

We begin with two elementary observations which will be used in the proof of Theorem \ref{radi}.

\begin{lemma}\label{potential}
Let $u\in L^1_{\rm loc}(\R^d)$ be a radially symmetric function,  
 such that  $v=E_d\ast u$ with $E_2(x)=-\frac1{2\pi}\log|x|$ and $E_d(x)=\frac1{(d-2)\sigma_d}|x|^{2-d}$ for $d\ge 3$, solves the Poisson equation $\Delta v+u=0$. Then $$\nabla v(x)\cdot x=-\frac1{\sigma_d}|x|^{2-d}\int_{\{|y|\le |x|\}} u(y)\dy.$$
\end{lemma}

\begin{proof}
By the Gauss formula, we have for the distribution function $M$ of $u$
$$M(R)\equiv \int_{\{|y|\le R\}} u(y)\dy=
-\int_{\{|y|=R\}} \nabla v(y)\cdot\frac{y}{|y|}{\rm d}S.$$ 
Thus,  for the radial function $\nabla v(x)\cdot\frac{x}{|x|}$ and $|x|=R$,  we obtain the identity
$$\nabla v(x)\cdot x=
\frac{1}{\sigma_d}R^{2-d}\int_{\{|y|=R\}}\nabla v(y)\cdot\frac{y}{|y|}{\rm d}S = -\frac1{\sigma_d}R^{2-d}M(R).$$
 \end{proof}

\begin{lemma}\label{MM}
If \ $\omega\in L^1_{\rm loc}(\R^d)$ is a radially symmetric function and $M(R)=\int_{\{|x|\le R\}} \omega(x)\dx$  --- its distribution function, then
$$
\int_{\{|x|\le R\}} \omega(x)M(|x|)\dx=\frac12M(R)^2.$$
\end{lemma}

\begin{proof}
Since $\omega$ is radial, it satisfies for $|x|=R$ the equality $\omega(x)=\frac{1}{\sigma_d}R^{1-d}M'(R)$. Thus, using the polar coordinates, we obtain
\bea
\int_{\{|x|\le R\}}\omega(x)M(|x|)\dx&=&\sigma_d\int_0^R\frac1{\sigma_d}r^{1-d}M'(r)M(r)r^{d-1}\dr\nonumber\\
&=&\int_0^R M'(r)M(r)\dr=\frac12M(R)^2.
\nonumber
\eea
\end{proof}

\begin{proof}[Proof of Theorem \ref{radi}.] We will derive a differential inequality for the quantity
\be
w_R(t)=\int \psi_R(x)u(x,t)\dx\label{moment}
\ee
with the scaled weight function $\psi_R$ supported on the ball $\{|x|\le R\}$ 
\be
\psi(x)=(1-|x|^2)^2\un_{\{|x|\le 1\}}\ \ \ {\rm and\ \ \ }\psi_R(x)=\psi\bigg(\frac{x}{R}\bigg)\ \ \ {\rm with\ \ \ }R>0.\label{psi}
\ee
The function $\psi\in{\mathcal C}^1(\R^d)$ has piecewise continuous and bounded second derivatives 
\begin{equation}\label{psi'}
\begin{split}
&\nabla\psi(x)=-4x(1-|x|^2)\un_{\{|x|\le 1\}}=-4x\,\psi(x)^{1/2},\\
&\Delta\psi(x)=(-4d+4(d+2)|x|^2)\un_{\{|x|\le 1\}}.
\end{split}
\end{equation}
Observe that $\psi$ satisfies the relation
\be
\Delta\psi(x)\ge-\frac{(d+2)^2}{2}\psi(x), \label{lappsi}
\ee
which is seen from the elementary inequality for the quadratic polynomial 
$$-4d+4(d+2)s\ge -\frac{(d+2)^2}{2}(1-s)^2,$$
 equivalent to $\big(s-\frac{d-2}{d+2}\big)^2\ge 0$,  applied to  $0\le s\le 1$.

Now, using equation \rf{equ}, integrations by parts and applying relations \rf{psi}--\rf{lappsi},  we obtain
\bea
\frac{{\rm d}}{\dt}w_R(t)&=& \int \Delta\psi_R(x) u(x,t)\dx +\int u(x,t)\nabla v(x,t)\cdot \nabla\psi_R(x)\dx\nonumber\\ 
&\ge& R^{-2} \bigg(-\frac{(d+2)^2}{2}\int \psi_R(x)u(x,t)\dx - 4\int u(x,t)\big(\nabla v(x,t)\cdot x\big)(\psi_R(x))^{1/2}\dx\bigg). \!\!\!\!\!\!\!\!\!\!\!\!\!\!\!\!\!   \label{evol}
\eea
Thus, by Lemma \ref{potential}, we get
\bea
R^2 \frac{{\rm d}}{\dt}w_R(t)&\ge& -\frac{(d+2)^2}{2}w_R(t) +\frac{4}{\sigma_d}\int u(x,t)M(|x|,t)|x|^{2-d}\psi_R(x)^{1/2}\dx\nonumber\\
&\ge&  -\frac{(d+2)^2}{2}w_R(t) +\frac{4}{\sigma_d}R^{2-d} \int \psi_R(x)u(x,t)M(|x|,t)\dx,\label{evol2} 
\eea
because $\psi_R(x)=0$ for $|x|\ge R$, and $\psi_R(x)\le \psi_R(x)^{1/2}$.

Now, note that obviously 
$$
M(R,t)=\int_{\{|y|\le R\}}u(y,t)\dy\ge\int_{\{|y|\le R\}}\psi_R(y)u(y,t)\dy.$$
Hence, applying Lemma \ref{MM} to the radial function $\omega(x)=\psi_R(x)u(x,t)$, we obtain
$$
\int\psi_R(x)u(x,t)M(|x|,t)\dx\ge \frac12\left(\int\psi_R(x)u(x,t)\dx\right)^2.
$$
 Thus, as a consequence of inequality \rf{evol2}, we arrive at
\be
R^2 \frac{{\rm d}}{\dt}w_R(t)\ge -\frac{(d+2)^2}{2}w_R(t)+\frac2{\sigma_d}R^{2-d}w_R(t)^2. \label{moment-ev}
\ee
Now, it is clear from \rf{moment-ev} that if
$$
R^{2-d}w_R(0)>\frac{(d+2)^2}{2}\left/\right.\left(\frac2{\sigma_d}\right) =(d+2)^2\frac{\sigma_d}{4}\equiv {C_d},
$$
then $-\frac{(d+2)^2}{2}w_R(0)+\frac2{\sigma_d}R^{2-d}w_R(0)^2\equiv \delta>0$.
Since the right-hand side of \rf{moment-ev} is an increasing function of $w_R$, we have $\frac{{\rm d}}{\dt}w_R(t)\ge \delta>0$. As a consequence, the function  $w_R(t)$ becomes greater than $M=\int u(x,t)\dx$ in a finite time which is  a~contradiction with the existence of nonnegative, mass conserving solutions.

Finally, observe that  if $d=2$ the conditions $M>C_2\equiv 8\pi$,  $\xn u_0\xn>8\pi$ and  $w_R(0)>8\pi$ for $R>0$ sufficiently large are equivalent. 

Similarly, if $d\ge 3$, the conditions $\xn u_0\xn>C_d$ and $R^{2-d}w_R(0)>C_d$ for some $R>0$ are equivalent. 
\end{proof}

\section{Blowup with $8\pi$ concentration of mass}\label{3}

The proof of Theorem \ref{symm}, saying that a solution to problem \rf{equ}--\rf{ini} on the whole plane ${\mathbb R}^2$  with $M>8\pi$ concentrates at the origin with mass not exceeding  $8\pi$ at the blowup time, is based on two auxiliary results: on a uniform estimate of the blowup time (Proposition \ref{genbl}) and on a uniformly slow spread of mass over annuli in $\R^2$ (Proposition \ref{symmass}).

\medskip

\subsection{Uniform blowup time}

In the following proposition, we show that the blowup time of a solution to problem \rf{equ}--\rf{ini} can be estimated from above by a number which depends only on an amount of $u_0$ concentrated in the unit ball. 

\begin{proposition}\label{genbl}
Let $\varepsilon>0$ and $\gamma>0$ be arbitrary and fixed.
Suppose that $u=u(x,t)$ is a solution of problem \rf{equ}--\rf{ini} with an initial datum satisfying 
$$
0\le u_0\in L^1(\R^2)\quad  \text{and}\quad   
\int_{\{|x|\leq \gamma\}}u_0(x)\dx \geq 8\pi +\e.
$$  
Then $u(x,t)$ blows up in a finite time $t=\tbl  \leq \gamma^2 T(M, \e)$, where  $T(M,\e)>0$ depends  on $M=\int u_0(x)\dx>8\pi$ and  $\e$, only.
\end{proposition}

\begin{remark}\label{rem:scal}
We introduce the parameter $\gamma>0$ in Proposition \ref{genbl} to simplify the notation in the proof.
In fact, assuming the property stated in the proposition  for $\gamma=1$, we obtain immediately this property for each other $\gamma>0$ by the rescaling $u\mapsto u_{\gamma^{-1}}=
\gamma^{-2} u(\gamma^{-1}x, \gamma^{-2} t)$.
\end{remark}

\begin{remark}\label{rem1}
Observe that for an initial condition $u_0$  with its support in the unit ball,  this proposition   holds true by the  standard second moment argument 
(cf. e.g. \cite{B-CM,BDP}) based on the identity 
$$
\frac{\rm d}{\dt}\left(\int|x|^2u(x,t)\dx\right)=\frac{1}{2\pi}M(8\pi-M)<0,
$$
which implies that a nonnegative solution $u(x,t)$ ceases to exist at a moment of time estimated from above by the number 
$2\pi \big(M(M-8\pi)\big)^{-1}\int |x|^2u_0(x)\dx$. 
Now, it suffices to notice that $\int|x|^2u_0(x)\dx\le M$ for ${\rm supp\,}u_0\subset \{|x|\le 1\}$ and choose $\e=M-8\pi$.
\end{remark}

\begin{remark}\label{rem:rad}
Proposition \ref{genbl}  has been already proved  in this paper in the radially symmetric case.
Indeed, it is sufficient  to apply inequality \rf{moment-ev} with $d=2$
and a suitable $R>0$:
\begin{equation} \label{wR2}
R^2 \frac{{\rm d}}{\dt}w_R(t)\ge -8w_R(t)+\frac1{\pi}
w_R(t)^2
\end{equation}
to the function $w_R$ defined by relations \rf{moment}--\rf{psi}.
We use this inequality with  $R=2^{1/2}(2+16\pi/\varepsilon)^{1/2}$.
By a direct calculation using the assumption on $u_0$, we obtain
$$
w_R(0)\ge \left(1-\frac{1}{R^2}\right)^2\int_{\{|x|\le 1\}}u_0(x)\dx 
\ge 
\left(1-\frac{1}{R^2}\right)^2(8\pi+\e)\ge 8\pi+\frac{\e}{2}.
$$ 
Thus, analogously as at the end of the proof of Theorem \ref{radi}, the function $w_R(t)$ becomes greater than $M$ in a finite time which can be estimated from above by a quantity depending on $M$ and $\e$, only.
\end{remark}

\begin{remark} \label{rem3}
We cannot directly apply the local moment method developed in \cite{BCKZ} to show Proposition \ref{genbl}  for general initial conditions, analogously as in the radial case discussed in Remark \ref{rem:rad}. 
This is due to the fact that blowup results in \cite{BCKZ} are proved for each initial datum $u_0$ such that $M=\int u_0(x)\dx >8\pi $ which, moreover, has a small mass outside a ball.
In fact, by methods of this work, we can remove that extra assumption from results proved in \cite{BCKZ}.
\end{remark}

One may summarize Remarks \ref{rem1}--\ref{rem3} by saying that the main problem in proving Proposition \ref{genbl} consists in controlling a large mass of a solution which is outside of the  unit ball.  
To show this proposition,  we study (as in the previous section)  the time evolution of the function
$$
w(t)=w_1(t)=\int \psi (x)u(x,t)\dx,
$$
and a solution $u(x,t)$ blows up at certain $\tbl$ if there exists $T\geq \tbl$ such that $w(T)=M$. 
Here, besides inequalities \rf{psi'}--\rf{lappsi}, we will use the following elementary estimates for the weight function $\psi=\psi_1(x)=(1-|x|^2)^2\un_{\{|x|\le 1\}}$:
\bea
|\psi(x)-1| &\leq& B|x|^2,\label{enjj}\\
|\nabla \psi(x)-\nabla \psi(y)+4(x-y)| &\leq & B\delta |x-y| {\rm  \ \ for\  all\ \  } |x|,\, |y|\leq \delta, \label{enjjj}\\
|(x-y)\cdot(\nabla \psi(x)-\nabla \psi(y))|&\leq & B\min\{|x-y|^2,|x-y|\},\label{en-}
\eea
valid for each fixed constant $0<\delta<1$, some constant $B\ge 1$ independent of $\delta$, and all $x$, $y\in\R^2$.

First, let us prove an auxiliary result concerning the function $w(t)$.

\begin{lemma}\label{nonum}
Given $\e \in (0, M-8\pi ]$, define the parameters
\begin{equation}\label{gal}
\eta=\eta(\e)=\frac{\e}{100M^2B}, \quad 
\alpha=\frac{1}{100MB}, \quad 
 \lambda=\lambda(\e)=\frac{100M^2B}{\e}+1,
\end{equation}
where $B\ge 1$ is a constant satisfying \rf{enjj}--\rf{en-}.
Assume that
\be
w(0)\geq 8\pi +\frac\e{2}.\label{startll}
\ee
Suppose that there exists $T\in (0,\tbl]$ such that for all $t\in [0,T)$ we have the estimate 
\be
\int_{\{\eta\leq |x|\leq \lambda\}}u(x,t)\dx < \e\alpha.\label{star}
\ee
Then, for all $t \in [0,T] $, the inequality  $\frac{{\rm d}}{\dt}w(t)\ge \frac1{100}\e$ holds true.
In particular,  the estimate \  $T\leq 100 M/\e$ follows.
\end{lemma}

\begin{proof}
Applying inequality \rf{en-} and then assumption \rf{star}, we have for our choice of $\eta$, $\alpha,\lambda$
\be
\int_{\{|y|\leq\eta\}}\int_{\{\eta \leq|x|\leq \lambda\}}
u(x,t)u(y,t)\frac{\big|(x-y)\cdot(\nabla \psi(x)- \nabla \psi(y)  )\big|}{|x-y|^2}\dx\dy\leq
BM\e\alpha= \frac{\e}{100},
\ee
and again by \rf{en-} since $\eta<1$
\be
\int_{\{|y|\leq\eta\}}\int_{\{\lambda \leq|x|\}}
u(x,t)u(y,t)\frac{\big|(x-y)\cdot(\nabla \psi(x)- \nabla \psi(y))\big|}{|x-y|^2}\dx\dy\leq
\frac{M^2B}{\lambda-1}= \frac{\e}{100}.
\ee
Moreover, by \rf{enjjj} and elementary calculations
{
\bea
&&\int_{\{|y|\leq\eta\}}\int_{\{|x|\leq\eta\}}
u(x,t)u(y,t)\frac{(x-y)\cdot(\nabla \psi(x)- \nabla \psi(y))}{|x-y|^2}\dx\dy \nonumber\\ 
&\ge& w(t)^2-2w(t)\int_{\{\eta\le|x|\le 1\}} \psi(x)u(x,t)\dx-\frac1{25}\e\nonumber\\
&\ge &w(t)^2-2M\e\alpha-\frac1{25}\e\nonumber\\
&=&w(t)^2-\frac1{50}\e-\frac1{25}\e\ge w(t)^2-\frac1{10}\e.\nonumber
\eea
}
Hence, repeating the above estimate with $x$ replaced by $y$,  and using \rf{lappsi} with $d=2$, we obtain
{ 
\bea
\frac{{\rm d}}{\dt}w(t)
&=&\int\psi(x)\Delta u(x,t)\dx-\frac{1}{2\pi}\iint u(x,t)u(y,t)\frac{x-y}{|x-y|^2}\cdot\nabla\psi(x)\dx\dy\nonumber\\
&=&\int \Delta\psi(x)u(x,t)\dx-\frac{1}{4\pi}\iint  u(x,t)u(y,t)\frac{x-y}{|x-y|^2}\dx\dy \cdot\left(\nabla\psi(x)-\nabla\psi(y)\right)\dx\dy\nonumber\\ 
&\geq& -8w(t)+\frac{1}{\pi}\left(w(t)^2 -\frac\e{10}\right)\nonumber\\
&=& \frac1\pi w(t)\big(w(t)-8\pi\big)-\frac{1}{10\pi}\e\ge \frac1{2\pi}\e-\frac{1}{10\pi}\e\ge \frac{1}{4\pi}\e\nonumber 
\eea
}
as long as $w(t)$ is increasing. Now, Lemma \ref{nonum} follows   by assumption \rf{startll}.

Since we cannot have the estimate  $w(T)> M$, integrating the differential inequality  $\frac{{\rm d}}{\dt}w(t)\ge \frac1{100}\e$ we obtain the upper bound  $T\leq 100M/\e$. 
\end{proof}

Now, we define a certain property of problem \rf{equ}--\rf{ini}. 

\begin{definition}\label{Iepsilon}
Fix $\e>0$ and $\gamma>0$.
Problem \rf{equ}--\rf{ini} is said to have {\it the property $\mathcal I_\e$} if each of its solutions corresponding to an initial datum satisfying 
\begin{equation}\label{Ie0}
0\le u_0\in L^1(\R^2)\quad  \text{and}\quad
\int_{\{|x|\leq \gamma\}}u_0(x)\dx\geq 8\pi +\e 
\end{equation}
blows up not later than at  time $\gamma^2 T(M, \e)$, with the parameter $T(M,\e)>0$ depending on $M=\int u_0(x)\dx>8\pi$ and $\e$, only. 
\end{definition}

Let us first notice elementary facts concerning the property $\mathcal I_\e$.

\begin{remark}\label{rem:Ie}
The parameter $\gamma>0$ can be easily removed from this definition  by the usual rescaling, cf. Remark \ref{rem:scal}.
Problem \rf{equ}--\rf{ini} has the property $\mathcal I_{M-8\pi}$, because then assumptions \rf{Ie0} with $\e= M-8\pi$ mean that $u_0$ is supported on the ball of radius $\gamma$. 
Hence, it suffices to apply Remark \ref{rem1}.
Obviously, there is no solution satisfying conditions \rf{Ie0} for $\e> M-8\pi$. 
It is also easy to show that if problem \rf{equ}--\rf{ini} has the property $\mathcal I_{\e}$, then it has the property $\mathcal I_{\tilde\e}$ for each $\tilde\e>\e$.
\end{remark}

\begin{proof}[Proof of Proposition \ref{genbl}.]
By Definition \ref{Iepsilon}, it suffices to show that problem \rf{equ}--\rf{ini} has   the property $\mathcal I_\e$ for all $\e>0$ and with a suitably chosen  $\gamma>0$. 
To do this, we are going to prove the following two claims for each $\e>0$ with  parameters $\alpha=\alpha(\e)$, $\eta$, and $\lambda(\e)$ defined in \rf{gal}.
\begin{itemize}
\item{\it Claim 1.} Suppose that $\e$ and $M$ satisfy the conditions
\begin{equation} \label{c1}
8\pi +\e (1+\eta^2\alpha/2)>M.
\end{equation}
Then problem \rf{equ}--\rf{ini}  has the property $\mathcal I_{\e}$.

\item{\it Claim 2.}  Suppose that 
\begin{equation} \label{c2}
8\pi +\e (1+\eta^2\alpha/2)\leq M,
\end{equation}
and problem \rf{equ}--\rf{ini} has the property $\mathcal I_{\e(1+\eta^2 \alpha/2)}$.
Then problem \rf{equ}--\rf{ini} has the property $\mathcal I_{\e}$.
\end{itemize}

Let us first prove that Claims 1 and 2 imply the property $\mathcal I_{\e}$ for all $\e>0$.
Obviously, inequality \rf{c1} holds true for $\e=M-8\pi$. Thus, since $ \alpha=\alpha(\e)$ is a continuous function of $\e$, inequality \rf{c1} holds true for all $\e\in (\e_0, M-8\pi]$ with some $\e_0<M-8\pi$, and problem \rf{equ}--\rf{ini} has the property $\mathcal I_{\e}$ in this range of $\e$ by Claim 1. Recalling  Remark \ref{rem:Ie}, define
$$
\e_0=\inf \{\e>0\;:\; \; \text{problem \rf{equ}--\rf{ini} has the property $\mathcal I_{\e}$}\},
$$
and suppose that $\e_0>0$. By continuity, there exists $\e_1>0$, such that 
$\e_1(1+\eta^2 \alpha(\e_1)/2)=\e_0$. For every $\tilde \e\in (\e_1, \e_0)$, we have the alternative: either $\tilde \e$ satisfies inequality \rf{c1} or inequality \rf{c2}. In both cases, either by Claim 1 or Claim 2,  problem  \rf{equ}--\rf{ini} has the property $\mathcal I_{\tilde \e}$. This is a contradiction with the definition $\e_0$ because $\tilde \e<\e_0$.

Now, we prove both Claims 1 and 2 simultaneously, and the scheme of the proof is the following. If assumption \rf{c1} is satisfied and if estimate \rf{star}  holds true for all $t\in [0,\tbl)$, the proof of Claim 1 is completed by Lemma \ref{nonum}.
At the first point $t=T_1$, where estimate \rf{star}  fails, we obtain inequality \rf{c2}. 
Hence, using the recurrence hypothesis of Claim 2 and a suitable rescaling of the whole problem, we obtain Claim 2.

Fix $\e \in (0,M-8\pi)$. Let $\eta=\eta(\e)$, $\alpha$, $\lambda(\e)$ be defined by \rf{gal}.
Set
\begin{equation}\label{gamma}
\gamma^2 = \frac{\alpha \eta^2\e}{2MB}
\end{equation}
and notice that $\gamma^2=\gamma^2(\e)\le 1$.
Suppose that $u_0$ satisfies conditions \rf{Ie0} with this value 
of $\gamma$.  
Thus, using inequality  \rf{enjj},  we obtain 
\begin{equation}\label{w(0)}
\begin{split}
w(0)&\ge \int_{\{|x|\le\gamma\}}\psi(x)u_0(x)\dx \\
&\ge \int_{\{|x|\le\gamma\}} u_0(x)\dx-\int_{\{|x|\le\gamma\}} |1-\psi(x)|u_0(x)\dx\\
&\ge 8\pi+\e- B\int_{\{|x|\le\gamma\}}|x|^2u_0(x)\dx\\
&\ge 8\pi+\e-B\gamma^2 M.
\end{split}
\end{equation}
Notice that, with our choice of $\gamma$ in \rf{gamma}, we have $MB\gamma^2<\e /2$, thus, we obtain the inequality $w(0)>8\pi +\e/2$, which is the first assumption \rf{startll} of Lemma \ref{nonum}.

Next, we deal with the second assumption \rf{star}  of Lemma \ref{nonum}.
Notice that if estimate \rf{star} holds true for all $t\in [0, \tbl)$ then,  by Lemma \ref{nonum}, we have the property $\mathcal I_{\e}$ with $\gamma$ defined in \rf{gamma}.

Suppose  that estimate  \rf{star} does not hold for $t=0$. Then, by assumption \rf{Ie0} and the inequalities $\gamma<\eta$ and $\eta^2\leq 1/2$, 
we obtain 
\begin{equation*}
\begin{split}
\int_{\{|x|\le\lambda\}}u_0(x)\dx
&=\int_{\{|x|\le\gamma\}}u_0(x)\dx +\int_{\{\gamma\le |x|\le\lambda\}}u_0(x)\dx\\
&\ge 8\pi+\e+\e\alpha\ge 8\pi+\e(1+\eta^2\alpha/2).
\end{split}
\end{equation*}
Notice that this inequality cannot be true under the condition \rf{c1} of Claim 1, because then the total mass  of $u_0$ would be greater than $M$.
Thus, we have inequality \rf{c2} assumed in Claim 2.  Suppose that the second assumption of Claim 2 is satisfied, namely, that each solution of  problem \rf{equ}--\rf{ini} with an initial datum satisfying \rf{Ie0} with $\e$ replaced by $\e (1+\eta^2 \alpha/2)$  blows up at time estimated from above by $ \lambda^2 T(M, \e (1+\eta^2 \alpha/2))$. Now, we rescale the solution as explained in Remark~\ref{rem:scal},  to see that it suffices to  choose $T(M,\e) =\gamma^{-2} \lambda^2 T(M,\e (1+\eta^2\alpha/2))$.
Since $\gamma$, $\lambda$ depend only on $M$, $\e$, we obtain the property $\mathcal I_{\e}$.

Now, consider the case when assumption \rf{star} of  Lemma \ref{nonum} is not satisfied for some $t\in (0, \tbl)$.
Thus, by continuity,  there exists $T_1\in (0,\tbl)$ such that strict inequality \rf{star} is satisfied for all $t\in [0, T_1)$, and for $t=T_1$ we have
\be
\int_{\{\eta\leq |x|\leq \lambda\}}u(x,T_1)\dx =
\e\alpha. \label{annulus}
\ee
Hence, by Lemma \ref{nonum}, the function $w(t)$ is increasing for $t\le T_1$, 
and by \rf{w(0)} we obtain 
\be
w(T_1)=\int \psi(x)u(x,T_1)\dx\ge w(0)\ge 8\pi+\e-\gamma^2MB.\label{two}
\ee
Now,  the estimate  
$$\un_{\{|x|\le \lambda\}}\ge \psi(x)+  \big(1-(1-\eta^2)^2\big)\un_{\{\eta\le|x|\le \lambda\}}
\ge 
\psi(x)+  \eta^2 \un_{\{\eta\le|x|\le \lambda\}}
$$ 
implies 
\bea
\int_{\{|x|\leq \lambda\}}u(x,T_1)\dx 
&\geq& \int\psi(x)u(x,T_1)\dx+\eta^2 \int_{\{\eta<|x|<\lambda\}} u(x,T_1)\dx\nonumber\\
&\geq& 8\pi +\e -\gamma^2 BM +\eta^2 \e \alpha
\ge 8\pi +\e(1+\alpha\eta^2/2), \nonumber
\eea
because $\gamma^2BM\leq \eta^2 \e \alpha/2$.

Notice again, as above, that condition \rf{c1} of Claim 1 cannot be true.
Thus, if we assume both conditions of Claim 2, then problem \rf{equ}--\rf{ini} with the initial condition $u(x,T_1)$ blows up not later than at time  $t= \lambda^2T (M, \e (1+\alpha\eta^2/2))$.
Then, problem \rf{equ}--\rf{ini}  with the initial condition $u_0$ blows up at time estimated from above by 
$$ 
\tilde T(M,\e)= T_1 +\lambda^2 T(M,\e (1+\alpha \eta^2/2))\leq 100M/\e +\lambda^2 T(M, 
\e (1+\alpha\eta^2/2)). 
$$
Now, it suffices to rescale the problem by choosing  $T(M,\e)=\gamma^{-2} \tilde T(M,\e)$. 
\end{proof} 

\medskip

\subsection{Uniform spread of  mass in the cases  of  $N$-symmetry of initial data}

We prove our second auxiliary result for $N$-symmetric nonnegative solutions with  sufficiently large $N$.

 \begin{proposition}  \label{symmass}
Let $0\le u_0\in L^1(\R^2)$ and $\int_{\{\delta \leq|x|\leq R\}}u_0(x)\dx \geq \e$ for some $\e>0$ and $0<\delta<R/4<\infty$. If $u_0$ is  an $N$-symmetric function  (see \rf{sym}) where the integer $N>0$ is  sufficiently large, i.e. $N\ge cM$ with a constant $c>0$ independent of $M$ and of $\delta$, then the solution $u$ of  problem \rf{equ}--\rf{ini}, as long as this exists,  satisfies for $t>0$
$$
\int_{\{\frac12\delta \leq|x|\leq 3R \}}u(x,t)\dx \geq \e\exp(-Ct),$$
where $C=C(M, \e, \delta)>0$ depends only on $M$, $\delta$, $\e$, 
and is independent of  $R$.
\end{proposition}

\begin{proof}
First, we assume that $\delta=1$ and  $1<R<\infty$. 
Consider a weight function $\phi:\R^+\to\R^+$ such that
\begin{displaymath}
\phi(s)=\left\{\begin{array}{ll}
0 & \textrm{if $0\le s\le \frac12$,}\\
(2s-1)^2 &\textrm{if $\frac12\le s\le \frac34$,}\\
1 & \textrm{if $1\le s\le R$},\\
\big(1-\frac{s-R}{2R}\big)^2 &\textrm{if $2R\le s\le 3 R$},\\
0 & \textrm{if $3R\le s$}.\end{array}\right.
\end{displaymath}    
Such a function $\phi$ can be chosen  increasing on $[0,R]$,  decreasing on $[R,3R]$, ${\rm supp\,}\phi\subset [1/2,3R]$, and  piecewise ${\mathcal C}^2$,  with its derivatives $\phi^{(k)}$ satisfying for $k=0,\, 1,\, 2$: 
\begin{equation}\label{p:est}
 |\phi^{(k)}(s)|\le \frac{C}{s^k}
\qquad  \text{with a constant $C$ {\it independent} of $R$ and of $s>0$.}
\end{equation}
We define, for the  function  $\Phi(x)=\phi(|x|)$,  the moment function of the solution $u$ by  $H(t)= \int \Phi(x) u(x,t)\dx$ that measures mass of $u$ contained in the annulus $\left\{\frac12\le |x|\le 3R\right\}$.

First, we present a particularly simple argument for radial solutions based on the identity from Lemma \ref{potential}.
 For the evolution of $H$  we have the differential inequality
\bea
\frac{{\rm d}}{\dt}H(t)
&=&\int \Phi(x)\Delta u(x,t)\dx +\int u(x,t)\nabla v(x,t)\cdot \nabla\Phi(x)\dx\nonumber\\
&= &\int \Delta\Phi(x)\,u(x,t)\dx +\int u(x,t)\big(\nabla v(x,t)\cdot x\big)\frac{\phi'(|x|)}{|x|}\dx\nonumber\\
&\geq &\int u(x,t)\bigg(\Delta\Phi(x) - \frac{1}{2\pi}M(|x|, t)\frac{|\phi'(|x|)|} {|x|} \bigg)\dx\nonumber\\
&\geq &-C(M)\int \Phi(x)u(x,t)\dx=-C(M)H(t)\label{H}
\eea
with a constant $C(M)$ independent of $R$. 
In the last inequality, we have used the bound
$$\Delta \Phi(x) -\frac{1}{2\pi}M\frac{1}{ |x|} |\phi'(|x|)|\geq -C(M)\Phi(x),$$
which is valid since $\Delta \Phi(x)\ge  C|x|^{-2}>0$ for $\frac12 \leq |x|\leq \frac34$ and  $\frac52 R  \leq |x|\leq 3 R$, and because of estimates \rf{p:est}.
Since by the assumption  we have $H(0)\geq \e>0$,
 the above inequality yields  the conclusion of Proposition \ref{symmass} for radial solutions.

Now, we prove Proposition \ref{symmass} under the  $N$-symmetry assumption.
 Our goal is to derive again a differential inequality of the form
 $
 \frac{{\rm d}}{\dt}H(t)\ge -CH(t)
 $
 with a suitably large constant $C$ depending only on  $M$ and $N$, as was done in \rf{H}. 
Let us emphasize here that the crucial consequence of the $N$-symmetry assumption consists in some cancellations in the bilinear (w.r.t. $u$) integral $\int u\nabla v\cdot\nabla\Phi$ appearing in the first line of formula \rf{H}. 

Let us decompose the gradient of the weight function $\Phi$ as 
\be
\nabla\Phi=\mu +\nu \label{decomp}
\ee
 with 
 \be
 {\rm supp\,}\mu\subset\left\{\frac12\le|x|\le2\right\},\label{mu}
 \ee
 and 
 \be
 {\rm supp\,}\nu\subset\{R\le|x|\le 3R\}.\label{nu} 
 \ee 
We write $\frac{{\rm d}}{\dt}H(t)=I_0+I_1$, where 
$I_0=\int \Delta\Phi(x)u(x,t)\dx$ 
and 
$$I_1=-\iint u(x,t)u(y,t)\frac{x-y}{|x-y|^2}\cdot\nabla\Phi(x)\dx\dy\equiv I_{1,\mu} +I_{1,\nu},$$
according to \rf{mu}--\rf{nu}.  
Further, given  $A\ge 4$, we decompose the integral  $I_{1,\mu}$ into the sum of integrals 
$$I_{1,\mu}=J_1+J_2+J_3$$ 
with the integration domains 
$$\left\{\frac14\le|x|\le A,\, |y|\ge A\right\}, \quad
\left\{\frac14\le|x|\le 2,\, |y|\le \frac14 \right\}, \quad
\left\{\frac14\le|x|\le A,\, \frac14\le |y| \le A\right\},
$$
respectively. 
Note that, in fact, ${\rm supp\,}\mu\subset\subset\left\{\frac14\le|x|\le 2\right\}$.

The  integrals $J_1$, $J_2$ will be estimated rather crudely.
Using the property \rf{mu} we have the following bound for the integral $J_1$:
\bea
J_1&=&-\int_{\{\frac{1}{2}\le|x|\le 2\}}\int_{\{|y|\ge A\}} u(x,t)u(y,t)\frac{x-y}{|x-y|^2}\cdot \mu(x)\dx\dy \nonumber\\
&\ge& -\frac{2C}{A} M\int_{\{\frac{1}{2}\le|x|\le 2\}}u(x,t)\dx\nonumber\\
&\ge& -\frac{2CM}{A}\left(\int_{\{\frac{1}{2}\le|x|\le 2\}}u\Delta\Phi +C\int_{\{\frac{1}{2}\le|x|\le 2\}}u\Phi\right),\nonumber
\eea
since for 
$|x|\le 2$ and $4\le A\le |y|$ the relation $|x-y|\ge\frac12A$ holds and, moreover, $1\le \Delta\Phi(x)+C\Phi(x)$ for each $\frac12\le|x|\le 2$. 
Finally, let us define $A=4CM$. 

Now, observe that for each $\r>0$ there exists a constant $C_\r$ such that the  both ``weights'' $|\mu|$ and $|\mu|^\frac12$ are bounded by $\Delta\Phi$ and $\Phi$, i.e. the inequality
\be
|\mu(x)|+|\mu(x)|^\frac12\le \r\Delta\Phi(x)+C_\r\Phi(x)\label{est-mu}
\ee
holds. This inequality, together with $|x-y|\ge\frac14$ and \rf{mu}, leads to the estimate
\bea
J_2&=& -\int_{\{\frac12\le|x|\le2\}}\int_{\{|y|\le\frac14\}} u(x,t)u(y,t)\frac{x-y}{|x-y|^2}\cdot\mu(x)\dx\dy\nonumber\\
&\ge& -4\r M\int_{\{\frac12\le|x|\le2\}}\Delta\Phi(x)u(x,t)\dx-{C_\r}M\int_{\{\frac12\le|x|\le2\}}\Phi(x)u(x,t)\dx\nonumber
\eea 
where we put $\r=\frac{1}{24M}$, $A=6CM$, and we obtain the required inequality. 

To exploit some  gain  from the symmetry assumption \rf{sym} (i.e. discover  some cancellations in the integral $J_3$), we decompose $J_3$ further into the  integrals $J_{3,1}$ and $J_{3,2}$ over the disjoint sets
$$
\Omega_1=\left\{\frac14\le |x|\le A,\, \frac14 \le |y|\le A,\, x\in\Gamma_y\right\}\ {\rm and\ }
\Omega_2=\left\{\frac14\le |x|\le A,\, \frac14 \le |y|\le A,\, x\notin\Gamma_y\right\},
$$
where $\Gamma_y=\left\{x: \left|\frac{x}{|x|}-\frac{y}{|y|}\right|\le \frac{2\pi}{N}\right\}$ is the sector determined by the direction of $y$. 
 Obviously, we have $x\in \Gamma_y\Leftrightarrow y\in\Gamma_x$.
Denote by $S\subset\R^4$ the support of the function $|\mu(x)|+|\mu(y)|$.  
Since $\Omega_1$ is symmetric and if $(x,y)\in S$ then either $\frac12<|x|<2$ or $\frac12<|y|<2$, so for the first integral we have the representation 
\bea
J_{3,1} &=& \iint_{\Omega_1} u(x,t)u(y,t)\frac{x-y}{|x-y|^2}\cdot\mu(x)\dx\dy\nonumber\\
&=& \frac12\iint_{\Omega_1} u(x,t)u(y,t)\frac{(x-y)\cdot(\mu(x)-\mu(y))}{|x-y|^2}\dx\dy\nonumber\\
&=&\frac12\iint_{\Omega_1\cap S}u(x,t)u(y,t)\frac{(x-y)\cdot(\nabla\Phi(x)-\nabla\Phi(y))}{|x-y|^2}\dx\dy.\nonumber
\eea
Therefore, we can estimate $|J_{3,1}|$ by  
\bea
|J_{3,1}|&\le & 2\int_{\{\frac14\le|x|\le A,\, \frac14\le|y|\le A,\, x\in \Gamma_y\}\cap S} u(x,t)u(y,t)\dx\dy \nonumber\\
&\le& C\frac{4}{N}M\int_{\{\frac12\le|x|\le 2\}}u(x,t)\dx,\nonumber\\
&\le& C\frac{4}{N}M \left(\int_{\{\frac{1}{2}\le|x|\le 2\}}u\Delta\Phi +C\int_{\{\frac{1}{2}\le|x|\le 2\}}u\Phi\right),\nonumber
\eea
since 
$$
\int_{\{\frac14\le |y|\le 2,\, y\in\Gamma_x\}}u(y,t)\dy\le C\frac{2}{N}M
\qquad
 \text{for $x\in{\rm supp\,}\mu,$}
$$ 
due to the $N$-symmetry assumption. 

For the estimate of the integral $J_{3,2}$, note that the weight function $\mu$ satisfies
$$
\frac{|(x-y)\cdot(\mu(x)-\mu(y))|}{|x-y|^2}\le \frac{|\mu(x)-\mu(y)|}{|x-y|}. 
$$
Moreover, we have  
$$|\mu(x)-\mu(y)|\le C|x-y|,\ \ \ |\mu(x)-\mu(y)|\le |\mu(x)|+|\mu(y)|,$$
so putting together those inequalities, we arrive at
$$
|\mu(x)-\mu(y)|\le C|x-y|^{\frac12}\left(|\mu(x)|^{\frac12}+|\mu(y)|^{\frac12}\right).
$$
Observe that, if $x\notin\Gamma_y$, $|x|,\, |y|\ge \frac12$,  then $|x-y|\ge \frac{1}{cN}$ for some $c>0$. 
Consequently, using the symmetry of $\Omega_2$, we obtain after splitting the integration domain into dyadic pieces 
\bea
|J_{3,2}|&\le& C\iint_{\{\frac12\le|x|\le A,\, \frac14\le|y|\le A,\, x\notin\Gamma_y\}}|x-y|^{-\frac12}\left(|\mu(x)|^\frac12+|\mu(y)|^\frac12\right) u(x,t)u(y,t)\dx\dy\nonumber\\
&\le&C \sum_{1\le L\le cN, {\rm dyadic}}\!\!\left(\frac{L}{N}\right)^{-\frac12}\!\!\!\!\!\!\!\!\!\!\int_{\{\frac12\le|x|\le A,\, \frac14\le|y|\le A,\, x\notin\Gamma_y,\, |x-y|\simeq \frac{L}{N}\}}\!\!\!\!\!\!\!\!\!\!\!\!\!\!  \!\!\!\!\!\!\!\!\!\!\!\!\!\!\!\! \!\!\!\!\!\!\!\!\!\!\!\!\! \int \left(|\mu(x)|^\frac12+|\mu(y)|^\frac12\right)u(x,t)u(y,t)\dx\dy \nonumber
\eea 
with a constant $C>0$.

Observe that by the $N$-symmetry property, for $\frac12 \le |x|\le 2$, we have the bound 
$$
\int_{\{|y|\ge \frac14,\, |x-y|\simeq \frac{L}{N}\}}u(y,t)\dy \le C\frac{L}{N}M.
$$
Applying this, we estimate each summand by 
$$2C L^{-\frac12}N^\frac12\int|\mu(x)|^\frac12 u(x,t)\dx\, M\frac{L}{N}=C M L^\frac12N^{-\frac12}\int|\mu(x)|^\frac12u(x,t)\dx.$$
Since $\sum_{1\le L\le cN,{\rm dyadic}}L^\frac12\simeq cN^\frac12$, the entire sum is bounded from above by an application of inequality \rf{est-mu} 
$$
CM\int|\mu(x)|^\frac12 u(x,t)\dx \le C\r M\int_{\{\frac12\le|x|\le 2\}} \Delta\Phi(x)u(x,t)\dx + C_\r M\int_{\{\frac12\le|x|\le 2\}} \Phi(x)u(x,t)\dx.
$$
Now, we put $\rho=\frac{1}{6CM}$ and adding the inequalities we obtain the desired estimate \ \ $-|I_{1,\mu}| \geq -\frac12 \int_{\{\frac12\le|x|\le 2\}} \Delta\Phi \, u  -C\int_{\{\frac12\le|x|\le 2\}} \Phi\, u$.

Similar considerations apply to the part containing the function $\nu$ in the decomposition \rf{decomp}, and we get \ \ $-|I_{1,\nu}| \geq -\frac12 \int_{\{R\le|x|\le 3R\}} \Delta\Phi\, u  -CR^{-2}\int_{\{R\le|x|\le 3R\}} \Phi\, u$. We note that for $R=4$ the proof is literally 
as the above, and the general case
$R \geq 4$ follows by scalling. Adding this inequalities, and using $\Delta\Phi(x)=0$ for $2 \le |x|\le R$,
we obtain the differential inequality
$$
\frac{{\rm d}}{\dt}H(t)\ge -CH(t),
$$
similarly as was in \rf{H}. 
Note that we established that the assumption $N\geq CM$ where $C$ is large enough (but it  does not depend on $M$) is sufficient in order to  Proposition \ref{symmass} holds for $\delta=1$.

To complete the proof for arbitrary  $0<\delta\le 1<R$, it suffices   to use  the rescaling defined in \rf{scale} with $\lambda=1/\delta$, and then to replace $\delta R$ by $R$. 
\end{proof}

\subsection{Proofs of the $8\pi$-mass concentration  results}

\begin{proof}[Proof of Theorem \ref{symm}.]  
The idea of the proof is easy.
Suppose  for contradiction that $\int_{\{|x|\leq R_j\}}u(x, t_j)\dx \geq 8\pi +\e$ for some sequences $t_j\nearrow \tbl$, $R_j\searrow 0$, and a positive $\e>0$. 
It suffices to construct a sequence $\{S_j\}$  such that 
$$
\int_{\{S_j<|x|\leq 3 R_j\}}u(x, t)\dx \geq \eta(\e)>0 \qquad \text{for all $t\in [t_j,\tbl)$},
$$
where $\eta(\e)>0$ is independent of $j$.
If the annuli $\{x\in\R^2\,:\, S_j<|x|<R_j\}$ are disjoint, then the total  mass $\int u(x,t)\dx$ tends to infinity as $t\to \tbl$, which is impossible.

Under the assumption that $\int_{\{|x|\leq R_j\}}u(x, t_j)\dx \geq 8\pi +\e$ for a positive $\e>0$ and some sequences $t_j\nearrow \tbl$, $R_j\searrow 0$, there exists a sufficiently large constant $L$ depending only on $M$ and $\e$, such that we have  $\int \psi_{LR_j}(x)u(x, t_j)\dx\geq 8\pi +\frac\e{2}$.
To simplify the notation,  we will denote from now on the radii $LR_j$ again by $R_j$.
\par
We also define a  sequence $0<S_j< R_j$ such that $w_{S_j}(t_j)= 8\pi +\frac\e{4}$, while  by our choice $w_{R_j}(t_j)\geq 8\pi +\frac\e{2}$.
Observe, that the rescaled function $u_j(x,t)=S_j^2u(S_jx, t_j+ S_j^2 t)$ is a solution of the system \rf{equ}--\rf{eqv}  with the same mass $M$, cf. \rf{scale}.
Moreover, its initial condition $u_j(x,0)= S_j^2u(S_jx, t_j)$ satisfies assumptions of Proposition \ref{genbl}, that is
$$
\int_{\{|x|\leq 1\}}u_j(x,0)\dx \geq
\int_{\{|x|\leq 1\}}\psi(x)u_j(x,0)\dx \geq 8\pi +\frac\e{4}.
$$
Hence, the solution $u_j(x,t)$ cannot be classical after (by the definition, it means that this blows up not later than) $T_{\rm w}$, where    $T_{\rm w}=T_{\rm w}(M, \e)$ can be chosen {\em independently} of $j$ by Proposition \ref{genbl}. 
Denote by $\tbl(j)\leq T_{\rm w}(M, \e)$ the blowup time of the solution  $u_j(x,t)$.
\par\noindent By the definition of $S_j$ and $R_j$ and \rf{psi}, we have
$$w_{R_j}(t_j)-w_{S_j}(t_j)\geq \frac\e{8}.$$
Passing to the rescaled solution $u_j(x,t)$, we obtain
$$
\int\bigg(\psi_{\frac{R_j}{S_j}}(x)-\psi(x)\bigg)u_j(x,0)\dx\geq  \frac\e{16}.
$$
Since we have by \rf{enjj} $\bigg|\psi_{\frac{R_j}{S_j}}(x)-\psi(x)\bigg|\leq C|x|^2$ for some  $C=C(B)>0$, we obtain for an appropriate choice of the constant $\delta=\delta(M, \e)$ (it suffices to choose $\delta^2(M,\e)={\e}/({16CM})$),  still independent of $j$, 
$$\int_{\big\{\delta \leq |x|\leq \frac{R_j}{S_j}\big\}}u_j(x,0)\dx \geq \frac\e{16}.$$
Applying Proposition \ref{symmass} we infer that for $t\leq \tbl(j)$ we have
$$\int_{\big\{\frac12\delta \leq |x|\leq 3\frac{R_j}{S_j}  \big\}}u_j(x,t)\dx\geq\beta\frac\e{8},$$
where the constant $\beta>0$ depends only on $T_{\rm w}(M, \e)$, $M$, $\e$, and consequently
$$\int_{\big\{\frac12\delta \leq |x|\leq 3\frac{R_j}{S_j } \big\}}u_j(x, \tbl(j))\dx\geq \beta\frac\e{8}$$
with the same constant $\beta $.
Scaling back to the original coordinates, we get at the blowup time $\tbl$ 
\be
\int_{\big\{\frac12\delta S_j \leq|x|\leq 3 R_j\big\}}u(x, \tbl)\dx\geq \beta\frac\e{8}.
\label{bigmass}
\ee
So, if we choose $R_{j+1}$ (passing, if necessary, to a subsequence) satisfying  $3R_{j+1}\leq \frac12\delta S_j$, so that the annuli $\left\{\frac12\delta S_j\le|x|\le 3R_j\right\}$ are disjoint, we infer that $u(x, \tbl)$ accumulates infinite mass. Indeed,  masses estimated in \rf{bigmass} (each bounded  from below by the same positive number) are distributed over disjoint annuli $\{\frac12\delta S_j\le |x|\le 3R_j\}$ ---  a~contradiction.        
\end{proof}

\begin{proof}[Proof of Corollary \ref{cor:symm}.]
First, recall the main result obtained in \cite{BZ}:
 if the initial condition satisfies $\int_Bu_0(x)\dx\leq 8\pi-\e$ for every ball $B$ of radius $\delta$,  then the solution of problem \rf{equ}--\rf{ini} exists at least on the interval $[0, T(M,\delta,\e)]$, where the number $T(M,\delta,\e)>0$ depends on $M$, $\delta$, $\e$, only.

Now, let $u$ be an $N$-symmetric  solution with properties assumed in Theorem \ref{symm}. 
Suppose that, for some fixed $\e>0$ and $\delta>0$, there exists a sequence $t_j\nearrow \tbl$ such that
\begin{equation}\label{***}
\int_{\{|x|\le\delta\}} u(x,t_j)\dx \leq 8\pi-\e
\qquad \text{for each $t_j$}.
\end{equation}
Using the $N$-symmetry property and \rf{***}, it is easy to show that there exists $\delta_1\in (0,\delta]$ such that 
\begin{equation*}
\int_{\{|x-x_0|\le\delta_1\}} u(x,t_j)\dx \leq 8\pi-\e
\qquad \text{for each $x_0\in\R^2$ and each $t_j$}
\end{equation*}
(use the $N$-symmetry if $|x_0|>\delta/2$ and \rf{***} if $|x_0|<\delta$).

Now, by \cite{BZ}, the solution of problem \rf{equ}--\rf{ini} with the initial condition $u(x,t_j)$, $\tbl-t_j<T(M,\delta_1,\e)/2$, exists at least on the interval $[0, T(M,\delta_1,\e)/2]$.
Thus, we have extended the solution beyond the blowup time $\tbl$ which is a contradiction. 
Therefore, we have proved that 
$$
\liminf_{t\to \tbl} \int_{\{|x|<\delta \}}u(x,t)\dx \geq 8\pi\qquad \text{for all $\delta>0$},
$$
which, together with the upper bound in Theorem \ref{symm}, completes the proof of this corollary.
\end{proof}

\end{document}